\documentclass[preprint,10pt]{elsarticle}

\usepackage{amsmath}
\usepackage{amssymb}
\usepackage{amsthm}
\usepackage{mathtools}
\usepackage{colortbl}
\usepackage{cases}
\usepackage{graphicx}
\usepackage{url}
\usepackage[hidelinks]{hyperref}
\usepackage{mathrsfs}
\usepackage{mathtools}
\usepackage[algo2e,ruled,linesnumbered,noend]{algorithm2e}

\usepackage{amsfonts}
\usepackage{epsfig}
\usepackage{epstopdf}
\usepackage[english]{babel}
\usepackage{mathtools}
\usepackage{graphicx}
\usepackage[font=small,labelfont=bf]{caption}
\usepackage{epstopdf}
\usepackage{indentfirst}
\usepackage[a4paper]{geometry}
\theoremstyle{plain}
\newtheorem{thm}{Theorem}[section]
\newtheorem{lem}[thm]{Lemma}
\newtheorem{prop}[thm]{Proposition}

\newtheorem{defn}{Definition}[section]
\newtheorem*{thm*}{Theorem}
\newtheorem*{lem*}{Lemma}
\newtheorem*{prop*}{Proposition}
\newtheorem*{cor*}{Corollary}
\newtheorem*{defn*}{Definition}
\theoremstyle{definition}

\theoremstyle{remark}
\newtheorem*{rem}{Remark}
\newtheorem*{oss}{Observation}

\begin{document}

\begin{frontmatter}

\title{Stationary solutions for dyadic mixed model of the Euler equation. A complete spectrum.}

\author[uniflo,isti]{Carlo Metta\corref{cor1}}

\ead{carlo.metta@isti.cnr.it}
\cortext[cor1]{Corresponding author.}

\address[uniflo]{{University of Florence. Department of Mathematics, Informatics and Statistics, Italy.}}
\address[isti]{{ISTI-CNR, Via G. Moruzzi, Pisa, Italy.}}
\begin{center}
\begin{abstract}
\begin{flushleft}
Dyadic models of the Euler equations were introduced as toy models to study the behaviour of an inviscid fluid in turbulence theory. In 1974 Novikov proposed a generalized mixed dyadic model that extends both Katz-Pavlovic and Obukhov models giving birth to a more complex structure: no results were found in literature until 2015 where blow up in finite time for smooth solutions and existence of self-similar solution for particular values of the model parameters were shown by Jeong I.J. We extend such partial results by giving a complete spectrum of existence and uniqueness results for two cardinal classes of finite energy stationary solutions, namely constant and self-similar solutions.
\end{flushleft}
\end{abstract}
\end{center}


\end{frontmatter}

\section{Introduction}

In this work we prove existence and uniqueness results for special classes of stationary solutions of the following system of ordinary differential equations:

\begin{equation}\label{mixedmodeliniziale}
\begin{split}
    \frac{d Y_n(t)}{dt} = \delta_1 &[k_n Y^2_{n-1}(t) - k_{n+1} Y_n(t) Y_{n+1}(t)] \\ -\delta_2 &[k_n Y^2_{n+1}(t) - k_{n-1} Y_n(t) Y_{n-1}(t)]
\end{split}
\end{equation}
for $n\geq 1$, where $k_n = 2^{\beta n}$ for some $\beta >0$, $\delta_1, \delta_2 \geq 0$ non negative parameters, and boundary conditions $Y_{-1} \equiv 0$, $Y_n(0) = y_n$.
\\
\\
Model (\ref{mixedmodeliniziale}) belongs to the family of dyadic shell models developed in turbulence theory in order to deepen the study of turbulent fluid dynamics \cite{biferale1}. These models are consistent with but simpler than the cardinal Euler equation. They are often used as toy model to experiment and apply novel and useful techniques related to Euler dynamics (for some effective results in this direction see for example \cite{bianchi2020}, \cite{tao1}).

In the Fourier representation of three-dimensional Euler equation, the transfer of energy from large to small scales is described as a flux of energy from small wave numbers to large wave numbers. The idea behind shell models is to divide the space into concentric spheres with exponentially growing radius $k_n = \lambda^n$, $\lambda > 1$. We then call $n$-th shell the set of wave numbers contained in the $n$-th sphere and not contained in the $(n-1)$-th sphere. The square of the scalar variable $Y_n(t)$ is linked with the energy of a fluid velocity vector field restricted to $n$-th shell frequency. The quantity
$$
E(t) = \sum_{n \geq 0} Y_n^2(t)
$$
is then called the \emph{energy} of the solution $Y_n(t)$. One can easily check from (\ref{mixedmodeliniziale}) that energy is \emph{formally} conserved in time. It is natural to restrict the study to non-negative and finite energy solutions.
\par
In order to state regularity results for dyadic models it is common to define counterpart of the Sobolev norms in the space of sequences. The $H^s$ norm of a solution $Y = (Y_0, Y_1, \ldots)$ at a certain time $t$ is given by the formula
$$
||Y_n(t)||_s^2 = \sum_{n\geq 0} 2^{2sn} Y_n^2(t).
$$
In particular, the energy of a solution is the square of its $H^0$ norm, usually called the $\ell^2$ norm.
\par
Shell models investigate the energy cascade flow with a system of coupled nonlinear ordinary differential equation of the form:

\begin{equation*}
    \frac{d}{dt} Y_n = k_n G_n [Y,Y] - \nu_n Y_n + f_n,
\end{equation*}
where the variable $Y_n$ represents the evolution over time of the velocity over a wavelength of scale $k_n$. The nonlinear function $G_n[\cdot,\cdot]$ is chosen to preserve some suitable properties inherited from the original nonlinear terms of Euler equation. Moreover, it is common for $G_n[\cdot, \cdot]$ to couple only scales that are close to each other (for instance nearest and next-to-nearest shells). The model is said to be either inviscid or viscous depending on whether the friction coefficient $\nu = \{ \nu_n \}_n $ is equal to 0. Similarly, the model is called unforced if there is no forcing term in the equation ($f_n = 0$), otherwise forced, if we add a positive forcing (constant) term to the first node in order to simulate a force sustaining turbulent fluids.
\par
The constraints to have local interaction, quadratic non-linearity, preserving total energy (or total helicity), and phase-space evolution do not fix uniquely the form $G_n[\cdot, \cdot]$. Consequently, many models have been developed even in recent years in order to study different aspects of turbulent fluids.
\\
\\
It is worth noting that model (\ref{mixedmodeliniziale}) reduces to what in the literature is called Katz-Pavlovic dyadic model \cite{katzpavlovic} (even though it had already been introduced by Novikov \cite{Novikov}) and Obukhov dyadic model \cite{obukhov2} by setting respectively $\delta_2 = 0$ and $\delta_1=0$, thus we expect it to carry both Katz-Pavlovic and Obukhov dynamics giving birth to a more complex structure: even simple uniqueness and positiveness properties proved for example in \cite{barbato2}, \cite{barbato1}, \cite{cheskidov1}, \cite{katzpavlovic}, \cite{kiselev} do not hold anymore.
\\
\\
As observed many times in literature, for example by Waleffe \cite{waleffe}, Katz Pavlovic and Obukhov models constitute the two basic blocks of all models satisfying four characteristic features derived from the spectral form of Euler equation: (i) quadratic non-linearity, (ii) appropriate scaling property of dimensionless coefficients, (iii) energy conservation, (iv) and nearest neighbor coupling. All of these except the last one descend from the Euler equation: the last condition is a simplification to make the problem more tractable. Thus, it is natural to investigate the dynamics of the generalized model (\ref{mixedmodeliniziale}), which from now on we refer to as the \emph{mixed (linear) dyadic model}, where we use the \emph{linear} attribute to distinguish from the more complex family of \emph{tree models} (see \cite{barbato4}, \cite{bianchi}, \cite{structurefunction}).
\\
\\
In Section \ref{risultati} we start by introducing main properties of model (\ref{mixedmodeliniziale}), focusing on existence and uniqueness results for \emph{constant} and \emph{self-similar} solutions. These results show a transition from a region, in the parameter space, where uniqueness is known to another one where it is false or open. They are consistent with and extend previous result found in literature. A K41 scaling property is also proved.
Section \ref{dimostrazione1} is devoted to present a different proof of Theorem 10 in \cite{barbato1}. The original proof was based on complex analysis and numerical simulations. Since it was not clear whether it can be extended to the case $\delta_1, \delta_2 >0$, we prefer to adopt another strategy by means of a novel pull-back technique. In Section \ref{dimostrazione2} and \ref{dimostrazione3} we then extend such technique proving the main results of Section \ref{risultati}, namely existence and uniqueness statements respectively for constant and self-similar solutions.

\section{Mixed dyadic model}\label{risultati}

In \cite{barbato2} positiveness property for weak solutions of the Katz-Pavlovic linear model was proved. This property plays a crucial role in many cardinal results, like the exponential global attraction of finite energy solutions to the unique constant solution (see \cite{cheskidov3}).
\\
Unfortunately, the positiveness property does not hold anymore in the mixed dyadic model. Indeed, by the variation of constants formula

\begin{equation}\label{constanformulabis}
\begin{split}
    Y_n(t) &= Y_n(t_0) \cdot e^{-\int_{t_0}^t [\delta_1 k_{n+1} Y_{n+1}(s) - \delta_2 k_{n-1} Y_{n-1}(s)]ds} + \\
    &+\int_{t_0}^t k_n [\delta_1 Y_{n-1}^2(s) - \delta_2 Y_{n+1}^2(s)]  \cdot e^{-\int_{s}^t [\delta_1 k_{n+1} Y_{n+1}(z) - \delta_2 k_{n-1} Y_{n-1}(z)]dz} ds.
    \end{split}
\end{equation}
In the Katz-Pavlovic model, i.e.\,\,$\delta_2 = 0$, one can easily derive that positiveness holds for positive initial condition. However, in the mixed model (\ref{mixedmodeliniziale}) the following condition
\begin{equation}\label{positivitamixed}
\delta_1 Y_{n-1}^2(t) - \delta_2 Y_{n+1}^2(t) \geq 0
\end{equation}
does not hold in general.
\\
\\
Because of its complex dynamics, there were no results in literature until 2015 \cite{jeong2}, where the authors show that smooth solutions blows up in finite time when $\delta_2 / \delta_1 $ is small enough, extending previous results which proved blow-up in the KP model \cite{cheskidov1}, \cite{frielander}, \cite{katzpavlovic}, \cite{kiselev}, \cite{waleffe}. Only later in 2019 \cite{jeong} the existence of self-similar solution for particular value of parameters $(\delta_1, \delta_2)$ was proved, within a local uniqueness theorem. Moreover, in \cite{montagnani}, the author proved the following theorem about the existence of weak solution of the mixed dyadic model for every initial condition $y \in \ell^2$, in the case $\delta_1 = \delta_2 = 1$.

\begin{thm}[Montagnani \cite{montagnani}]\label{teomonta}

Consider the infinite dimensional shell model
 $$\frac{d}{dt} Y_n (t) = k_n Y_{n-1} ^2 (t)- k_{n+1} Y_n (t) Y_{n+1} (t) - k_n Y_{n+1} ^2 (t) + k_{n-1} Y_n (t) Y_{n-1} (t)  ,$$
 $$ Y(0) = y.$$ 
 
Then, for any initial condition $y \in \ell^2$ there exists at least one solution $Y(t)$ on $[0,T]$.
\end{thm}

Extending Theorem \ref{teomonta} to general parameters $\delta_1, \delta_2$ is straightforward.
\\
\\
In the next sections we prove existence and uniqueness results for both constant and self-similar solutions, for every couple of parameter $(\delta_1, \delta_2)$ in the positive quadrant. A K41 scaling property is also proved. These findings are consistent with the partial results found in \cite{jeong}, although we prefer to consider model (\ref{mixedmodeliniziale}) as originally expressed in \cite{Novikov2} and \cite{kiselev}.

\subsection{Constant solutions}

We recall that a \emph{constant} solution $Y= \{Y_n(t)\}$ of (\ref{mixedmodeliniziale}) is a solution that is time independent, i.e. $Y_n(t) = a_n$ for all $t \geq 0$  and  some $a_n \in \mathbb{R}_{\geq 0}$. From the point of view of fluid dynamics theory it is common to restrict ourselves to study positive constant solutions with the additional condition of being finite energy, i.e.
$$
\sum_{n=1}^{\infty} a_{n}^2 < \infty.
$$
In order to consider (\ref{mixedmodeliniziale}) as a model of turbulence dynamics, it is natural to add a constant forcing term ($F > 0$) to sustain turbulent regime:

\begin{equation}\label{condizioneiniziale}
    \frac{d Y_0(t)}{dt} = -\delta_1 k_{1} Y_0(t) Y_{1}(t) -\delta_2 Y^2_{1}(t) + F.
\end{equation}
In particular, if $Y_n(t) = a_n$ is a constant solution, then 

\begin{equation}\label{stationarysolution}
    \delta_1 [k_n a_{n-1}^2 - k_{n+1}a_n a_{n+1}] = \delta_2 [k_n a_{n+1}^2 - k_{n-1}a_n a_{n-1}],
\end{equation}
and the first terms $a_0, a_1$ are related to $F$ by the following relation

\begin{equation}\label{condizioneiniziale}
    \delta_1 k_{1} a_0 a_1 +\delta_2 a_1^2 = F.
\end{equation}
Despite their simplicity, constant solutions are one of the most important class of solutions. In the case $\delta_2=0$ there exists a unique constant solution of the form

\begin{equation*}
    a_n = C_F \cdot k_n^{-1/3}, \,\,\,\,\, n \geq 0,
\end{equation*}
where $C_F$ is a constant depending on the starting force $F$ and the scaling $k_n^{-1/3}$ is reminiscent of the famous Kolmogorov K41 \cite{KolmogorovK41} law as well. Moreover, the existence of a constant solution in a forced system is an example of anomalous dissipation in turbulent fluids. In \cite{cheskidov3}, the authors showed that, given a forcing term $F$, the only constant solution is an exponential global attractor for every finite energy solution.
\\
The existence of a global attractor for an inviscid system is, perhaps, surprising. However, it is perfectly consistent with the concept of anomalous or
turbulent dissipation conjectured by Onsager \cite{onsager}. It is then natural to investigate whether constant solutions continue to exist in the mixed model and how the competition of Katz-Pavlovic and Obukhov dynamics affects the behaviour of such solutions.

Later we prove the following result:

\begin{thm}\label{stationary}
The forced mixed model $(\ref{mixedmodeliniziale})$ admits positive constant solutions for every choice of coefficient $\delta_1, \delta_2>0$.
\\
\\
In particular:
\begin{itemize}
    \item if $\frac{\delta_1}{\delta_2} < k_1^{-4/3}$, then for every $F>0$ there exist infinitely many positive finite energy constant solution $\{ a_n \}_{n\geq0}$, one for every choice of $a_0>0$;
    \item if $\frac{\delta_1}{\delta_2} > k_1^{-4/3}$, then for every $F>0$ there exists just one value $a_0>0$ from which stems a positive finite energy constant solution $\{ a_n \}_{n\geq0}$.
\end{itemize}
Moreover, any such solution satisfies Kolmogorov's scaling law 
$$
\lim_{n \rightarrow \infty} \frac{a_n}{k_n^{-1/3}} = C
$$
for some positive constant $C>0$.
\end{thm}

Theorem \ref{stationary} offers a complete picture of how the two opposite dynamics compete against each other to shape the behaviour of fixed points (constant solutions). If the total impulse of Obukhov dynamics we assign to the model does not exceed $k_1^{4/3}$ times the corresponding Katz-Pavlovic dynamics, then existence and uniqueness of fixed point is preserved. It is an interesting open question whether this unique fixed point still attracts all the other finite energy solutions like in \cite{cheskidov3}.
On the other hand, if $\frac{\delta_1}{\delta_2} < k_1^{-4/3}$ it is not surprising that all constant solutions become finite energy since the regularization phenomenon related to Obukhov dynamics (\cite{kiselev}) prevails over Katz-Pavlovic.

\subsection{Self-similar solutions}

When the forcing term is absent no non-trivial fixed point are yet known in literature for Katz-Pavlovic and Obukhov model. However, there exists another special class of stationary solution closely related to constant solutions.

\begin{defn}
A \textbf{self-similar} solution is a finite energy solution $Y$ such that there exists a differentiable function $\phi(t) $ and a sequence of real numbers $a = (a_n)_{n \geq 1}$ such that $Y_n(t) = a_n \cdot \phi (t)$ for all $n \geq 1$ and all $t \geq 0$.
\end{defn}

\begin{oss}
It is easy to observe that positive self-similar solutions satisfying Katz-Pavlovic model have the form
\begin{equation}\label{selfsimilar}
Y_n(t) = \frac{a_n}{t-t_0},
\end{equation}
for some $t>t_0$ and $t_0<0$.
\\
Indeed, if a positive solution is of the form (\ref{selfsimilar}), then
\begin{equation*}
\begin{split}
-\frac{a_n}{(t-t_0)^2} = \frac{dY_n(t)}{dt} =& k_{n-1}Y_{n-1}^2(t) - k_nY_n(t)Y_{n+1}(t) \\ =& k_{n-1}\frac{a_{n-1}^2}{(t-t_0)^2} - k_n \frac{a_na_{n+1}}{(t-t_0)^2},
\end{split}
\end{equation*}
that leads us to the sequence $\{a_n\}_n $ satisfying
$$
a_n a_{n+1} = \frac{a^2_{n-1}}{2^{\beta}} + 2^{-\beta n}a_n.
$$
Although it is possible for the first terms $a_1, a_2, \ldots, a_{n_0}$ to be zero, if $a_{n_0+1}>0$ then all the subsequent coefficients must be positive too:
\begin{equation}\label{recursion}
    a_{n+1} = 2^{-\beta n} + \frac{a_{n-1}^2}{2^{\beta} a_n} > 0, \,\,\,\,\,\,\forall n \geq n_0 + 1.
\end{equation}

\end{oss}
Thus, without loss of generality, one can look for solutions with $a_0=0$ and $a_n>0$. Theoretically, every choice of $a_1>0$ would give a self-similar solution following recursion (\ref{recursion}), but we are only interested in finite energy solution. Unexpectedly, in \cite{barbato1} it was proved the existence of just one special $a_1>0$ that satisfies this condition in the Katz-Pavlovic setting.

\begin{thm}{\text{(Barbato, Flandoli, Morandin \cite{barbato1}})}\label{teoselfsimilar}
Given $t_0 < 0$, there exists a unique positive self-similar solution with $a_1 \neq 0$. Moreover, given $t_0 < 0$ and $n_0 \geq 0$, there exists a unique positive self-similar solution with
\begin{equation*}
    a_1 = a_2 = \cdots a_{n_0} = 0, a_{n_0+1} > 0.
\end{equation*}
In addition, the coefficients $a_n$ have the property
$$
\lim_{n \rightarrow \infty} \frac{a_n}{k_n^{-1/3}} = C_{n_0},
$$
for some constant $C_{n_0}$.
\end{thm}
Theorem \ref{teoselfsimilar} was originally proved by using complex analysis argument and with the help of numerical computation, and it was not clear whether such argument can be extended to the mixed dyadic model. In the next section we present a different proof based on a pullback technique that will be useful in other proofs.
\\
Theorem \ref{teoselfsimilar} shows also that the Kolmogorov scaling law appears in these special solutions, phenomenologically associated to decaying turbulence. But it is an open problem to understand whether all other solutions approach the self-similar ones and in which sense.
\\
The existence of finite energy self-similar solutions is of theoretical interest
in itself, in comparison with analogous investigations for Euler and Navier-Stokes equations, moreover the existence of such solutions has a number of implications. For instance, they realize perfectly the decay rate $t^{-1}$, coherently with Theorem 8 and 9 in \cite{barbato1}. It has been conjectured that the set of all self-similar solutions (set depending on $t_0 \in \mathbb{R}$ and $n_0 \geq 0$) 
attracts all other finite energy solutions. If this is the case, the decay rate $t^{-1}$ would be the true one for all solutions.
Moreover, self-similar solutions offer an easy example of lack of uniqueness as shown by next observation.

\begin{oss}
It is possible to prove that for some initial conditions in $H=\ell^2$ with all negative components there exist infinitely many finite energy solutions.
\\
Indeed, by Theorem \ref{teoselfsimilar}, there exists a self-similar solution $Y$ whose total energy is strictly decreasing. Let $T>0$, then $X(t) = -Y(T-t)$ is a local solution on $[0,T]$ by a simple time inversion property. For any time $s \in [0,T]$, let's consider the solution $X^s$ obtained by attaching $X$ on $[0,s]$ to a Leray-Hopf solution on $[s, \infty)$ with initial condition $X(s) = -Y(T-s) \in H$. The energy of this solution strictly increases on $[0,s]$ and then is non-increasing on $[s, \infty)$. Thus, to different values of $s$ correspond finite energy solutions which are really different, but all with the same negative initial condition $-Y(T)$.
\end{oss}

Once more, it is natural to address the question of existence and possibly uniqueness of self-smilar solutions also for the mixed dyadic model (\ref{mixedmodeliniziale}).
As observed previously, it is still easy to prove that self-similar solutions in the unforced ($F=0$) mixed model have the form
\begin{equation}\label{selfsimilarbis}
Y_n(t) = \frac{a_n}{t-t_0}, \,\,\,\,\, a_0 =0,
\end{equation}
with $t>t_0$ and $t_0<0$.
\\
If a positive solution is of the form (\ref{selfsimilarbis}), then
\begin{equation}\label{selfsimilarsolution}
\begin{split}
-\frac{a_n}{(t-t_0)^2} = \frac{dY_n(t)}{dt} =& \frac{\delta_1}{(t-t_0)^2}(k_{n}a_{n-1}^2 - k_{n+1}a_n a_{n+1}) \\ -& \frac{\delta_2}{(t-t_0)^2}(k_{n}a_{n+1}^2 - k_{n-1}a_n a_{n-1}),
\end{split}
\end{equation}
that leads us to the sequence $\{a_n\}_{n \geq 1} $ satisfying
$$
-\frac{a_n}{k_n} = \delta_1(a_{n-1}^2 - k_1 a_n a_{n+1}) - \delta_2(a_{n+1}^2 - k_1^{-1}a_n a_{n-1}).
$$
It is still possible for the first terms $a_1, a_2, \ldots, a_{n_0}$ to be zero, although if $a_{n_0+1}>0$ then all the subsequent coefficients must be not zero: indeed, from the latter relation if $a_{n_0} = 0$ then
\begin{equation*}
    \begin{split}
        -\frac{a_{n_0 +1}}{k_{n_0 + 1}} &= \delta_1 (-k_1 a_{n_0+1} a_{n_0 + 2}) - \delta_2 (a_{n_0 + 2}^2) \iff \\
        \frac{a_{n_0 +1}}{k_{n_0 + 1}} &= \delta_2 a_{n_0 + 2}^2 + \delta_1 k_1 a_{n_0+1} a_{n_0 + 2}
    \end{split}
\end{equation*}
and $a_{n_0 + 1} >0$ implies $a_{n_0 + 2} \neq 0$. Since we are interested in positive solutions, without loss of generality one can set $a_0 = 0$ and $a_{n}>0$ for every $n \geq 1$.
\\
In \cite{jeong} self-similar solutions were found for $\delta_1=1$ and small enough $\delta_2$, within a local uniqueness theorem.
In the next section we prove the following result:

\begin{thm}\label{teoselfsimilarmix}
Given $t_0 < 0$, and $ k_1^{-4} \leq \delta_1 / \delta_2 \leq 1$, there exist self-similar solutions of the unforced ($F=0$) model (\ref{mixedmodeliniziale}). In particular
\begin{itemize}
    \item if $ k_1^{-4} \leq \delta_1 / \delta_2 < k_1^{-4/3}$ then for every $a_1 >0$ there exists a self-similar solution $\{ a_n \}_{n \geq 0}$ of (\ref{mixedmodeliniziale});
    \item if $ k_1^{-4/3} < \delta_1 / \delta_2 \leq 1$ there exists just one self-similar solution $\{ a_n \}_{n \geq 0}$ of (\ref{mixedmodeliniziale}).
\end{itemize}

In addition, any such self-similar solution satisfies Kolmogorov's scaling law 
$$
\lim_{n \rightarrow \infty} \frac{a_n}{k_n^{-1/3}} = C
$$
for some positive constant $C>0$.

\end{thm}

Theorem \ref{teoselfsimilarmix} divides the positive plane in four subregions: above the line $\delta_1 / \delta_2 = k_1^{-4}$ and below $\delta_1 / \delta_2 = 1$ Theorem \ref{teoselfsimilarmix} does not give any information about existence of self-similar solutions; between the lines $\delta_1 / \delta_2 = k_1^{-4}$ and $\delta_1 / \delta_2 = k_1^{-4/3}$ we have existence but not uniqueness; between the lines $\delta_1 / \delta_2 = k_1^{-4/3}$ and $\delta_1 / \delta_2 = 1$ we have existence and uniqueness of self-similar solution.
\\
However, as we will see later, upper and lower bounds for the ratio $\delta_1 / \delta_2$ can be further refined. Numerical simulations suggest the existence of a \emph{true} bound $L_{true}< k_1^{-4}$ such that Theorem \ref{teoselfsimilarmix} holds in the wider domain  $\delta_2 \cdot L_{true} \leq \delta_1$. This result is consistent with and complements Theorem \ref{teoselfsimilar} and Theorem 1 in \cite{jeong} by giving a full spectrum of behaviour of self-similar solution in the mixed dyadic model.
Once again, we stress that when the impulse $\delta_2$ given by Obukhov dynamics exceeds $k_1^{4/3}$ times the opposite impulse $\delta_1$, the regularization phenomenon occurs and all self-similar solutions become finite energy. On the contrary, when the Katz-Pavlovic dynamics prevails in the sense of second part of Theorem \ref{teoselfsimilarmix}, then the dynamics resembles the one described by Theorem \ref{teoselfsimilar} where just one finite energy self-similar solution exists.

\section{Proof of Theorem \ref{teoselfsimilar}}\label{dimostrazione1}

The original proof of Theorem \ref{teoselfsimilar} was based on complex analysis and numerical simulations, while in \cite{jeong} another proof was presented based on the analysis of certain dynamical systems on the plane. Since it was not clear whether such result could be fully extended to the whole parameters positive plane $(\delta_1, \delta_2)$, we prefer to adopt another strategy by means of a novel methodology of pull-back suggested in \cite{structurefunction} for studying constant solution of more structured dyadic models. Pull-back is a key dynamical idea often used in the literature to construct special solutions of differential equations, like periodic solutions. However, its utility is usually associated to the non-autonomy of certain systems: for example when coefficients are time-dependent and one is looking for infinite-time objects, the natural ways is to start from minus infinity. It is extremely interesting that this technique is successful not just in the Katz-Pavlovic framework, completely autonomous, but also in the generalized mixed dyadic model, as we will prove in last sections.
\\
\\
As already mentioned in (\ref{recursion}), we are interested in the following recursion

\begin{equation}\label{recursionbis}
    a_{n+1} = 2^{-\beta n} + \frac{a_{n-1}^2}{2^{\beta} a_n} > 0, \,\,\,\,\,\,\forall n \geq 1, \beta>0,
\end{equation}
where, without loss of generality, we set $a_0=0$ and $a_1>0$. Our main goal is to prove the existence of one and only one $a_1>0$ such that the self-similar solution that stems from $(a_0, a_1)$ following the rule (\ref{recursionbis}) has finite energy.

In order to lighten the notation, we prove Theorem \ref{teoselfsimilar} in the case $\beta = 1$. The general case will be a straightforward consequence.
\\
\\
First of all we start by proving the following proposition. 
\begin{prop}\label{teofacile}
Consider the recursion
\begin{equation}\label{epsilonsegniqualsiasi}
\begin{split}
    a_{n+1} &= \frac{a^2_{n-1}}{2a_n} + \epsilon_n, \,\,\, \epsilon_n = 2^{-n}, \,\,\, n\geq 1\\
    a_0 &= 0.
\end{split}    
\end{equation}
There is one and only one $u = \{ u_n\}_{n \in \mathbb{N}}$ positive finite energy sequence satisfying (\ref{epsilonsegniqualsiasi}). Moreover, such  $\{ u_n\}$ lies in $H^s$ for any $s < 1/3$. 
\end{prop}
\begin{proof}
We structure our proof in three different steps. At first, we prove the existence of solution for recursion (\ref{epsilonsegniqualsiasi}), then we show regularity of such solution, finally we prove uniqueness among positive solutions with finite energy.
\\
\\
\textbf{Step (1)}: \textbf{Existence}.
\\
\\
We start by considering the following two definition.

\begin{defn}
We call \textbf{strong self-similar} any positive sequence $\{ a_n \}_{n \in \mathbb{N}}$ satisfying the recurrence:
\begin{equation}\label{strongselfsimilar}
\begin{split}
    a_{n+1} &= \frac{a_{n-1}^2}{2a_n} + \epsilon_n, \,\,\, \epsilon_n = 2^{-n}, \,\,\, n\geq 1\\
    a_0 &= 0.
\end{split}
\end{equation}
\end{defn}

\begin{defn}
We call \textbf{weak self-similar} any positive sequence $\{ \tilde{a}_n\}_{n \in \mathbb{N}}$ satisfying the recurrence:
\begin{equation}\label{weakselfsimilar}
\begin{split}
    \tilde{a}_{n+1} &= \frac{\tilde{a}_{n-1}^2}{\tilde{a}_n} + \zeta_{n}, \,\,\, \zeta_n = \epsilon_n \cdot 2^{\frac{n-2}{3}}, \,\,\, n\geq 1\\
    \tilde{a}_0 &= 0.
\end{split}    
\end{equation}
\end{defn}

\begin{rem}
It is easy to verify that if $\{ \tilde{a}_n\}_{n \in \mathbb{N}}$ is a weak self-similar sequence then $\{ a_n = \frac{\tilde{a}_n}{2^{n/3}}\}_{n \in \mathbb{N}}$ is a strong self-similar sequence.
Conversely, for any strong self-similar sequence it is possible to recover the corresponding weak sequence from the equality above.
\end{rem}

In order to prove existence of strong self-similar sequence we use a pull back technique: we first consider recursion (\ref{weakselfsimilar}) backwards fixing a large $N>2$ and two appropriate starting values $\tilde{a}_{N+1}$ and $\tilde{a}_N$, then compute $\tilde{a}_n$ for lower coefficients $n<N$; finally we let $N \rightarrow \infty$ proving convergence by compactness to a weak self-similar solution and finally recovering a strong self-similar sequence from the remark above.
\\
\\
Thus, for any fixed $N>2$ we are interested in the following truncated reversed recursion:
\begin{equation}\label{backwardself}
\begin{split}
    &(\tilde{a}_{n-1}^{(N)})^2 = \tilde{a}_n^{(N)} (\tilde{a}_{n+1}^{(N)} - \zeta_{n}), \,\,\, n \leq N, \,\,\,\ \zeta_{n} = \epsilon_n \cdot 2^{\frac{n-2}{3}}\\
    &\tilde{a}_{N+1}^{(N)} = \tilde{a}_{N}^{(N)} = L > 0,\\
    &\tilde{a}_n^{(N)} = 0, \,\,\, n > N+1,
\end{split}
\end{equation}
where the initial value L will be chosen later accordingly to our requirements.
\\
\\
By taking a close look at recursion (\ref{backwardself}) we observe that the right-hand side must be non negative for every $n \leq N$. Thus, we first have to assess under which conditions recursion (\ref{backwardself}) is well-defined. The following proposition poses sufficient conditions for existence of weak self-similar solution.

\begin{prop}\label{existence}
Consider the system of equations (\ref{backwardself}) and let $M >0$ be such that
$$
M = \sum_{i=1}^{\infty} \zeta_i < \infty.
$$
Then any initial value $L \geq M$ gives rise to a well-defined sequence $\{ \tilde{a}_n^{(N)}\}$ for every $N>2$.
\end{prop}
\begin{proof}

We start by proving that every $\{ \tilde{a}_n^{(N)}\}$ that satisfies (\ref{backwardself}), when well-defined, is weakly increasing. We proceed by induction on $n$.

First two base cases are easy to verify:
$$
\tilde{a}_{N+1}^{(N)} = L = \tilde{a}_{N}^{(N)},
$$
$$
\tilde{a}_{N-1}^{(N)} = \sqrt{\tilde{a}_N^{(N)} (\tilde{a}_{N+1}^{(N)} - \zeta_{n-1})} = \sqrt{L (L - \zeta_{n-1})} < L = \tilde{a}_{N}^{(N)}.
$$
For the inductive step, we consider by hypothesis
\begin{equation*}
    \tilde{a}_{n+2}^{(N)} \geq \tilde{a}_{n+1}^{(N)} \,\,\, (i), \,\,\, \tilde{a}_{n+3}^{(N)} \geq \tilde{a}_{n+2}^{(N)} \implies \tilde{a}_{n+3}^{(N)} - \zeta_{n+1} \geq \tilde{a}_{n+2}^{(N)} - \zeta_{n} \,\,\,(ii)
\end{equation*}
and multiplying together inequalities $(i)$ and $(ii)$ we get:
$$
(\tilde{a}_{n+1}^{(N)})^2 = \tilde{a}_{n+2}^{(N)} (\tilde{a}_{n+3}^{(N)} - \zeta_{n+1}) \geq \tilde{a}_{n+1}^{(N)} (\tilde{a}_{n+2}^{(N)} - \zeta_{n}) = (\tilde{a}_{n}^{(N)})^2,
$$
proving the claim $\tilde{a}_{n+1}^{(N)} \geq \tilde{a}_{n}^{(N)}$.
\\
\\
It is also immediate to verify by induction that $\{ \tilde{a}_n^{(N)}\} \leq L$ for every $n \leq N$.
\\
\\
Let us consider again the increasing property in the following form:
$$
(\tilde{a}_{n+1}^{(N)})^2  \geq (\tilde{a}_{n}^{(N)})^2 = \tilde{a}_{n+1}^{(N)} (\tilde{a}_{n+2}^{(N)} - \zeta_n),
$$
and dividing both sides for the positive term $\tilde{a}_{n+1}^{(N)}$ we finally get
$$
\tilde{a}_{n+2}^{(N)} - \tilde{a}_{n+1}^{(N)} \leq \zeta_n.
$$
Applying a recursive argument to the inequality above it is possible to show
$$
\tilde{a}_{N}^{(N)} - \tilde{a}_{1}^{(N)} \leq \sum_{i=1}^{N-1} \zeta_i \leq \sum_{i=1}^{\infty} \zeta_i = M.
$$
We finally deduce that any initial value L satisfying
$$
0 \leq L - M \leq \tilde{a}_{n}^{(N)} \leq L
$$
gives rise to a well-defined truncated sequence. In particular, it is sufficient that $L \geq M$, concluding the proof.
\end{proof}

Proposition \ref{existence} tells that for every $N>2$, $\{\tilde{a}_{n}^{(N)}\}_n$ lies in the compact set $[L-M, L]$, thus by compactness and a diagonal extraction argument we can choose a subsequence $(N_i)_i \in \mathbb{N}$ such that $\tilde{a}_{n}^{(N_i)}$ converges for all $n \in \mathbb{N}$ to some number $\tilde{a}_{n}$. The sequence $\tilde{a} = \{ \tilde{a}_n \}_n$ satisfies recursion (\ref{weakselfsimilar}) by construction. In order to say that $\{ \tilde{a}_n \}_n$ is a weak self-similar sequence it is left to prove that $\tilde{a}_0=0$. The next Proposition shows how to choose the starting value $L$ in order to achieve this goal.

\begin{prop}\label{proptecnica}
Let us consider recursion (\ref{backwardself}) and a subsequence $(N_i)_i \in \mathbb{N}$ such that $\tilde{a}_{n}^{(N_i)}$ converges to $\tilde{a}_{n}$. There exists a starting value $L>0$ such that $\tilde{a}_{n}$ is weak self-similar, in particular
\begin{equation}
    \tilde{a}_0 = \lim_{N_i \rightarrow \infty} \tilde{a}_0^{(N_i)} = 0
\end{equation}
\end{prop}

\begin{proof}
Proposition \ref{existence} proves that, for every $N_i>2$, recursion $\tilde{a}_n^{(N)}$ is monotonously increasing. Thus, the idea is to the set the starting value $L$ as low as possible to produce a suitable recursion that proves the claim. With this regard, up to extracting one more subsequence of $(N_i)$, for every fixed $N_i>2$ consider the following
\begin{equation*}
I^{(N_i)} = \inf \{ L>0 \,\, | \,\, \text{the recursion that starts at} \,\, \tilde{a}_{N_i}^{(N_i)}=\tilde{a}_{N_i-1}^{(N_i)}=L \,\, \text{is well-defined}\}
\end{equation*}
Thanks to Proposition \ref{existence}, $I^{(N_i)}$ exists and $0 < I^{(N_i)}\leq M$. Moreover, the sequence $I^{N_i}$ is non-decreasing in $N_i$. Indeed, if $N_i<N_j$ then the sequences $\tilde{a}_{n}^{(N_i)}$, $\tilde{a}_{n}^{(N_j)}$ built from the same starting value satisfy $\tilde{a}_{n}^{(N_i)} \geq \tilde{a}_{n}^{(N_j)}$
thanks to the monotonic property. From the latter follows that the minimum starting value from which stems a well-defined $(N_i)$-truncated recursion is less or equal than the equivalent for an $(N_j)$-truncated recursion.
\\
We then claim that from the following starting value
\begin{equation*}
    L^{*} = \sup_{N_i>2} I^{(N_i)}
\end{equation*}
stems a limit sequence $\tilde{a}_n$ that is weak self-similar.
\\
First, we observe that $L^*$ is finite since $I^{(N_i)}\leq M$ for every $N_i>2$. The sequence $\tilde{a}_n$ satisfies recursion (\ref{backwardself}), hence the only property left to prove is
\begin{equation*}
    \tilde{a}_0 = \lim_{N_i \rightarrow \infty} \tilde{a}_0^{(N_i)} = 0.
\end{equation*}
On the contrary, suppose $\tilde{a_0}> \epsilon > 0$ for some $\epsilon>0$ and consider $f^{(N_i)}:\mathbb{R}_{>0}\rightarrow \mathbb{R}_{\geq 0}$ defined by
$$
f^{(N_i)}(L) = \tilde{a}_0^{(N_i)}.
$$
The function $f$ maps a positive number $L$ into the first element of the truncated $(N_i)$ recursion with starting values identical to $L$. By hypothesis, there is an index $N^*$ such that $f^{(N_i)}(L^*)>\epsilon>0$ for every $N_i>N^*$. Hence, by definition of $L^*$ and a continuity argument, there is a small enough value $\zeta>0$ and an index $N^{\zeta}>N^*$ such that $f^{(N_i)}(L^*-\zeta)>\epsilon>0$ for every $N_i>N^{\zeta}$. Since $I^{N_i}$ is monotonous, the latter would imply that
$$
L^* = \sup_{N_i>2} I^{N_i} \leq L^* - \zeta, 
$$
that is an absurd.

\end{proof}

Proposition \ref{proptecnica} proves that $\{ \tilde{a}_n \}_n$ is a weak self-similar sequence, and $a = \{a_n = \frac{\tilde{a}_n}{2^{n/3}}\}_n$ is the corresponding strong self-similar sequence.\\
\\
Furthermore, we observe that the condition
$$
M = \sum_{i=1}^{\infty} \zeta_i < \infty
$$
is equivalent to
$$
\sum_{n=1}^{\infty} \epsilon_n \cdot 2^{n/3}  = \sum_{n=1}^{\infty} 2^{-2n/3} < \infty,
$$
as required by Proposition \ref{teofacile}.
\\
\\
\textbf{Step (2)}: \textbf{Regularity}.
\\
\\
We are now ready to prove that any strong self-similar sequence $a = \{ a_n\}_{n \in \mathbb{N}}$ has finite energy, i.e.
$$
\sum_{n=1}^{\infty} a_n^2 < \infty.
$$
Moreover, such  $\{ a_n\}$ lies in $H^s$ for any $s < 1/3$.

\begin{prop}\label{L}
For every $s < 1/3$ any well-defined strong self-similar sequence has finite energy and lies in $H^s$.
\end{prop}
\begin{proof}
In Step (1) we have already shown that any weak self-similar sequence built from $L$ satisfies
$$
0 \leq L - M \leq \tilde{a}_{n} \leq L.
$$
By recovering the correct expression for the related strong self-similar sequence, we derive
$$
\sum_{n=1}^{\infty} 2^{2sn} \cdot a^2_n = \sum_{n=1}^{\infty} 2^{2sn} \cdot \frac{\tilde{a}_n}{2^{2n/3}} \leq L \cdot \sum_{n=1}^{\infty} 2^{n(2s-2/3)}.
$$
Finally, from the latter equation it follows that any strong self-similar sequence built from $L$ lies in $H^s$ for every $s < 1/3$.
\end{proof}

\textbf{Step (3)}: \textbf{Uniqueness}.
\\
\\
We now prove uniqueness among strong self-similar sequence with finite energy.

\begin{prop}\label{unicita}
Let $\{ b_n \}_n$ be a solution of recursion (\ref{recursion}) different from $\{ a_n \}_n$. Then exists $\alpha > 0$ such that for every $n \geq 3$:
\begin{equation*}
    \begin{split}
        &b_n \geq a_n \cdot 2^{\alpha n},\,\,\,\,\, \,\,\,\, \,n \,\,\,odd \\
        &b_n \leq a_n \cdot 2^{-\alpha n},\,\,\,\,\,  \,n \,\,\,even
    \end{split}
\end{equation*}
or
\begin{equation*}
\begin{split}
        &b_n \geq a_n \cdot 2^{\alpha n},\,\,\,\,\,\,\,\,\,  \,n \,\,\,even \\
        &b_n \leq a_n \cdot 2^{-\alpha n},\,\,\,\,\,  \,n \,\,\,odd
    \end{split}
\end{equation*}
\end{prop}
\begin{proof}
We prove only the first case of the proposition, the second being similar. We start by considering odd values of $n$ and finally even values.
\\
\\
\textbf{Case (1)}: $n$ odd.
By induction over $n$. By hypothesis $\{ b_n \}_n$ is different from $\{ a_n \}_n$, so without loss of generality we can suppose $b_3 > a_3$ and the existence of a real number $\alpha_1>0$ so that
$$
b_3 \geq a_3 \cdot 2^{3 \alpha_1},
$$
moreover, by Definition \ref{teofacile}
$$
b_4 = \frac{b_2^2}{2b_3} + \epsilon_3 < \frac{a_2^2}{2a^3} + \epsilon_3 = a_4,
$$
thus there exists also a real number $\alpha_2 >0$ so that
$$
b_4 \leq a_4 \cdot 2^{-4 \alpha_2},
$$
finally, by setting $\alpha = \min\{\alpha_1, \alpha_2\}$ we have proved base cases of induction.
\\
If $n$ is an odd number, then by hypothesis we have
\begin{equation*}
    b_{n+1} = \frac{b^2_{n-1}}{2b_n} + \epsilon_n \geq \frac{2^{2\alpha (n-1)}a^2_{n-1}}{2^{-\alpha n}(2a_n)} + \epsilon_n.
\end{equation*}
In what follows we will show that
\begin{equation*}
    \frac{2^{2\alpha (n-1)}a^2_{n-1}}{2^{-\alpha n}(2a_n)} + \epsilon_n \geq 2^{\alpha(n+1)}(\frac{a^2_{n-1}}{2a_n} + \epsilon_n) = 2^{\alpha(n+1)}a_{n+1},
\end{equation*}
concluding the proof.
\\
Let us first rewrite the latter inequality in the more compact form
\begin{equation}\label{compactform}
    (a_{n+1} - \epsilon_n) \cdot (2^{3\alpha n - 2\alpha} - 2^{\alpha n + \alpha}) \geq \epsilon_n \cdot (2^{\alpha n + \alpha} - 1).
\end{equation}
We structure the proof in two different steps.
\\
\\
\textbf{Step (1)}: $a_{n+1} \geq \epsilon_{n-1}$.
\\
\\
Let's rewrite the claim in terms of the corresponding $\tilde{a}_n$ weak sequence:
$$a_{n+1} \geq \epsilon_{n-1} \Longleftrightarrow \frac{\tilde{a}_{n+1}}{2^{(n+1)/3}} \geq \epsilon_{n-1} \Longleftrightarrow \tilde{a}_{n+1} \geq 2^{\frac{-(2n-4)}{3}}.
$$
In the latter inequality the left side in an increasing function of $n$ while the right side is decreasing, thus it is enough to verify the inequality for the smallest meaningful value of odd $n$, i.e. $n=5$:
$$
\tilde{a}_{6} \geq 2^{-2}.
$$
In what follows we are going to prove the stronger relation:
$$
\tilde{a}_{4} \geq 2^{-2}.
$$
The latter is equivalent to $a_{4} \geq 2^{-\frac{10}{3}}$ for the related strong self-similar sequence. We recall that without loss of generality we have chosen $a_0=0$ and $a_1 \in \mathbb{R}$ and, by a direct calculation, thanks to (\ref{strongselfsimilar}) we derive
$$
a_4 = \frac{1}{8} \cdot \frac{a_1^2 + \frac{3}{2}}{a_1^2 + \frac{1}{2}}.
$$
Finally
$$
a_4 \geq 2^{-\frac{10}{3}} \iff \frac{1}{8} \cdot \frac{a_1^2 + \frac{3}{2}}{a_1^2 + \frac{1}{2}} \geq 2^{-\frac{10}{3}} \iff \frac{a_1^2 + \frac{3}{2}}{a_1^2 + \frac{1}{2}} \geq 2^{-\frac{1}{3}},
$$
with the last inequality been true for every choice $a_1 \in \mathbb{R}$.
\\
\\
\textbf{Step (2)}: $ (2^{3\alpha n - 2\alpha} - 2^{\alpha n + \alpha}) \geq (2^{\alpha n + \alpha} - 1)$.
\\
\\
First we rewrite above inequality in the form
$$
2^{3\alpha n - 2\alpha} +1 \geq 2^{\alpha n + \alpha + 1},
$$
then observing that both sides are increasing function of $n$ and left sides grows faster than right sides, it is again enough to prove the claim for the smallest meaningful value $n=5$.
\\
By letting $n=5$ we obtain:
\begin{equation*}
\begin{split}
2^{13\alpha} - 2^{6 \alpha +1} + 1 = (2^\alpha - 1)\cdot [&(2^{12 \alpha} - 2^{5\alpha}) + (2^{11 \alpha} - 2^{4 \alpha}) + (2^{10 \alpha} - 2^{3 \alpha}) + \\ &(2^{9 \alpha} - 2^{2 \alpha}) + (2^{8 \alpha} - 2^{\alpha}) + (2^{7 \alpha} - 1) + 2^{6 \alpha}] > 0
\end{split}
\end{equation*}
because it is a product of positive numbers.
By multiplying together inequalities in Step (1) and Step (2) one can derive (\ref{compactform}).
\\
\\
\textbf{Case (2)}: $n$ even.
By induction over $n$. In the previous case we have already shown that exists a real number $\alpha > 0$ so that
$$
b_3 \geq a_3 \cdot 2^{3 \alpha}, \,\,\, b_4 \leq a^4 \cdot 2^{-4 \alpha}.
$$

If $n$ is an even number, then by hypothesis we have
\begin{equation*}
    b_{n+1} = \frac{b_{n-1}^2}{2b_n} + \epsilon_n \leq \frac{2^{-2\alpha(n+1)} \cdot a_{n-1}^2}{2^{\alpha n}a_n} + \epsilon_n.
\end{equation*}
We will now show that
$$
\frac{2^{-2\alpha(n+1)} \cdot a_{n-1}^2}{2^{\alpha n} \cdot a_n} + \epsilon_n \leq 2^{-\alpha(n+1)} (\frac{a_{n-1}^2}{2a_n} + \epsilon_n) = 2^{-\alpha(n+1)}a_{n+1}.
$$
concluding the proof.

First, we rewrite inequality above as follows:
\begin{equation}\label{tesi}
(a_{n+1} - \epsilon_n) (2^{-3\alpha n - 2\alpha} - 2^{-\alpha n - \alpha}) \leq \epsilon_n (2^{-\alpha n - \alpha} - 1).
\end{equation}

In the previous case we have already shown that $a_{n+1} \geq \epsilon_{n-1}$, thus as a fortiori argument we have the following:
\begin{equation*}
    2^n a_{n+1} \geq \epsilon_{n-1}.
\end{equation*}
Moreover, observing that for every $\alpha > 0$
$$
(2^{-3\alpha n - 2\alpha} - 2^{-\alpha n - \alpha}) < 0, \,\,\,\,(2^{-\alpha n - \alpha} - 1) < 0,
$$
in order to prove (\ref{tesi}) it is enough to show that
$$
2^{-n}(2^{-3\alpha n - 2\alpha} - 2^{-\alpha n - \alpha}) \geq  (2^{-\alpha n - \alpha} - 1),
$$
or equivalently
$$
2^{3\alpha n + 2\alpha + n} +1 \geq  2^{2\alpha n + \alpha}(2^n +1).
$$
Both sides are increasing functions of $n$ and left side increases faster than right side, so it is enough to prove the claim for the smallest admissible even $n$, i.e. $n=6$. Namely, we need to prove
$$
2^{20\alpha+6} +1 \geq  2^{13\alpha+6} + 2^{13\alpha}.
$$
Let's consider the function
$$f(\alpha) = 2^{20\alpha+6} +1 -  2^{13\alpha+6} - 2^{13\alpha}.$$
It is easy to notice that $f(0) = 0$ and $f(x)$ has positive derivative on the positive x-asis, namely
$$
\frac{df}{dx} = 5\cdot 2^{13} \cdot \log_2 (2^{7x + 8} - 169) > 0, \,\,\,\, x\geq 0,
$$
this proves the claim.
\end{proof}
Proposition \ref{unicita} tells us that every solution $b_n$ different from $a_n$ cannot have finite energy. Moreover, any other solution except $a_n$ cannot lie in any space $H^s$ even for negative values of $s$.
\end{proof}

We now observe that Theorem \ref{teoselfsimilar} is an immediate consequence of Proposition \ref{teofacile}.

\section{Proof of Theorem \ref{stationary}}\label{dimostrazione2}

This section is entirely devoted to proving Theorem \ref{stationary}.
\\
\\
Let us start by considering the following recursive sequence:
\begin{equation}\label{ricorrenzaavanti}
\begin{split}
&b_0 = C > 0, \\
&b_{n+1} = \frac{-\delta_1 k_1^{4/3} + \sqrt{\delta_1^2 k_1^{8/3} + 4\delta_1  \delta_2 k_1^{4/3} b_n^{-2} + 4\delta_2^2 b_n^{-1}}}{2\delta_2}
\end{split}
\end{equation}
for some positive starting value $C>0$ and some positive coefficient $\delta_1, \delta_2$ such that $\delta_1 / \delta_2 < k_1^{-4/3}$.
\\
Lemma \ref{lemmaavanti} tells useful information about the sequence $\{ b_n \}_n$ and its asymptotic behaviour.

\begin{lem}\label{lemmaavanti}
For every starting value $C>0$, and positive coefficient $\delta_1, \delta_2$ such that $\delta_1 / \delta_2 < k_1^{-4/3}$, the recursive sequence (\ref{ricorrenzaavanti}) satisfies
$$
\lim_{n \rightarrow \infty} b_n = 1.
$$
\end{lem}
\begin{proof}
Since recursion (\ref{ricorrenzaavanti}) admits $1$ as unique fixed point, we first observe that if $C=1$ then $b_n \equiv 1$ for every $n\geq 0$.
\\
\\
Without loss of generality let us suppose $C<1$ (the case $C>1$ being specular).
We will prove the following properties
\begin{enumerate}
    \item $b_{2n+1}>1, b_{2n} < 1$, $\forall n\geq 0$;
    \item $1< b_{2n+1}<b_{2n-1}$ and $0<b_{2n} < b_{2n+2} < 1$, $\forall n\geq 0$;
    \item $\lim_{n \rightarrow \infty} b_{2n+1} = \lim_{n \rightarrow \infty} b_{2n} = 1$,
\end{enumerate}
the statement will follow trivially.
\\
\\
We start observing that $b_1>1$:
\begin{equation*}
    \begin{split}
        b_1 &= \frac{-\delta_1 k_1^{4/3} + \sqrt{\delta_1^2 k_1^{8/3} + 4\delta_1 \delta_2 k_1^{4/3} C^{-2} + 4\delta_2^2 C^{-1}}}{2\delta_2}\\
        &> \frac{-\delta_1 k_1^{4/3} + \sqrt{\delta_1^2 k_1^{8/3} + 4\delta_1 \delta_2 k_1^{4/3} + 4\delta_2^2 }}{2\delta_2} = 1.
    \end{split}
\end{equation*}

Let's now suppose $b_{2n-1}>1$ for some $n>0$. Then we have
$$
b_{2n+1} = \frac{-\delta_1 k_1^{4/3} + \sqrt{\delta_1^2 k_1^{8/3} + 4\delta_1 \delta_2 k_1^{4/3}b_{2n}^{-2} + 4\delta_2^2 b_{2n}^{-1}}}{2\delta_2},
$$
moreover, by inductive hypothesis and definition
\begin{equation*}
\begin{split}
b_{2n} =& \frac{-\delta_1 k_1^{4/3} + \sqrt{\delta_1^2 k_1^{8/3} + 4\delta_1 \delta_2 k_1^{4/3}b_{2n-1}^{-2} + 4\delta_2^2 b_{2n-1}^{-1}}}{2\delta_2}\\
< &\frac{-\delta_1 k_1^{4/3} + \sqrt{\delta_1^2 k_1^{8/3} + 4\delta_1 \delta_2 k_1^{4/3} + 4\delta_2^2}}{2\delta_2} = 1,
\end{split}
\end{equation*}
hence $b_{2n}^{-1} > 1$, and finally 
\begin{equation*}
\begin{split}
b_{2n+1} =& \frac{-\delta_1 k_1^{4/3} + \sqrt{\delta_1^2 k_1^{8/3} + 4\delta_1 \delta_2 k_1^{4/3}b_{2n}^{-2} + 4\delta_2^2 b_{2n}^{-1}}}{2\delta_2}\\
&> \frac{-\delta_1 k_1^{4/3} + \sqrt{\delta_1^2 k_1^{8/3} + 4\delta_1 \delta_2 k_1^{4/3} + 4\delta_2^2}}{2\delta_2} = 1,
\end{split}
\end{equation*}
proving property (1).\\
We now focus on the first part of property (2) (the second being identical).
By definition we can write
\begin{equation*}
    \begin{split}
        b_{2n+1} < b_{2n-1} \iff &\frac{-\delta_1 k_1^{4/3} + \sqrt{\delta_1^2 k_1^{8/3} + 4\delta_1 \delta_2 k_1^{4/3}b_{2n}^{-2} + 4\delta_2^2 b_{2n}^{-1}}}{2\delta_2} < b_{2n-1}\\
        \iff & 4\delta_1 \delta_2 k_1^{4/3}b_{2n}^{-2} + 4\delta_2^2 b_{2n}^{-1} < 4b_{2n-1}^2\delta_2^2 + 4\delta_1 \delta_2 b_{2n-1} k_1^{4/3}\\
        \iff & \delta_1 k_1^{4/3}b_{2n}^{-2} + \delta_2 b_{2n}^{-1} < b_{2n-1}^2\delta_1 + \delta_2 b_{2n-1} k_1^{4/3}\\
        \iff & \delta_1 k_1^{4/3} (b_{2n}^{-2} - b_{2n-1}) < \delta_2 (b_{2n-1}^2 - b_{2n}^{-1}).
    \end{split}
\end{equation*}
By hypothesis, we set $\delta_1 k_1^{4/3} < \delta_2$, thus it is enough to require
$$
b_{2n}^{-2} - b_{2n-1} < b_{2n-1}^2 - b_{2n}^{-1} 
$$
within the positive condition on the right side $0 < b_{2n-1}^2 - b_{2n}^{-1}$.
We observe that the above two inequalities are both satisfied if $b_{2n}^{-1} < b_{2n-1}$, indeed:
$$
b_{2n}^{-1} < b_{2n-1} \implies b_{2n}^{-1} + b_{2n}^{-2}< b_{2n-1} + b_{2n-1}^2
$$
and 
$$
b_{2n}^{-1} < b_{2n-1} \implies b_{2n}^{-2} < b_{2n-1}^2 \implies 0 < b_{2n}^{-1} < b_{2n}^{-2} < b_{2n-1}^2,
$$
the latter being true due to $b_{2n} < 1$.
\\
\\
We are now left to prove the sufficient condition $b_{2n}^{-1} < b_{2n-1}$ or equivalently
\begin{equation*}
    \begin{split}
&b_{2n}^{-1} = \frac{2 \delta_2}{-\delta_1 k_1^{4/3} + \sqrt{\delta_1^2 k_1^{8/3} + 4\delta_1 \delta_2 k_1^{4/3}b_{2n-1}^{-2} + 4\delta_2^2 b_{2n-1}^{-1}}} < b_{2n-1}\\
& \iff \frac{2 \delta_2}{b_{2n-1}} + \delta_1 k_1^{4/3} < \sqrt{\delta_1^2 k_1^{8/3} + 4\delta_1 \delta_2 k_1^{4/3}b_{2n-1}^{-2} + 4\delta_2^2 b_{2n-1}^{-1}}\\
& \iff 4\delta_1^2 b_{2n-1}^{-2} + 4\delta_1 \delta_2 k_1^{4/3} b_{2n-1}^{-1} < 4\delta_1 \delta_2 k_1^{4/3}b_{2n-1}^{-2} + 4\delta_2^2 b_{2n-1}^{-1}\\
&\iff \delta_2 b_{2n-1}^{-2} + \delta_1 k_1^{4/3} b_{2n-1}^{-1} < \delta_1 k_1^{4/3}b_{2n-1}^{-2} + \delta_2 b_{2n-1}^{-1}\\
& \iff \delta_2 + \delta_1 k_1^{4/3} b_{2n-1} < \delta_1 k_1^{4/3} + \delta_2 b_{2n-1}\\
& \iff (\delta_2 - \delta_1 k_1^{4/3}) < (\delta_2 - \delta_1 k_1^{4/3})\cdot b_{2n-1},
    \end{split}
\end{equation*}
finally the latter inequality holds because $ \delta_2 - \delta_1 k_1^{4/3}>0$ and $b_{2n-1} > 1$.\\
We can now say that $b_{2n+1}$ admits limit $\lim_{n \rightarrow \infty} b_{2n+1} = L \geq 1$. Suppose $L > 1$, then

\begin{equation}\label{equazioneL}
L = \frac{-\delta_1 k_1^{4/3} + \sqrt{\delta_1^2 k_1^{8/3} + 4\delta_1 \delta_2 k_1^{4/3} g(L)^{-2} + 4\delta_2^2 g(L)^{-1}}}{2\delta_2}
\end{equation}
where 
$$
g(L) = \frac{-\delta_1 k_1^{4/3} + \sqrt{\delta_1^2 k_1^{8/3} + 4\delta_1 \delta_2 k_1^{4/3} L^{-2} + 4\delta_2^2 L^{-1}}}{2\delta_2}
$$
By a direct calculation equation (\ref{equazioneL}) holds if and only if
\begin{equation*}
    \begin{split}
        & 4\delta_2^2L^2 + 4\delta_1 \delta_2 k_1^{4/3}L= 4\delta_1 \delta_2 k_1^{4/3} g(L)^{-2} + 4\delta_2^2 g(L)^{-1}\\
        & \iff \delta_2 L^2 + \delta_1 k_1^{4/3}L= \delta_1 k_1^{4/3} g(L)^{-2} + \delta_2 g(L)^{-1}\\
        & \iff \delta_2 (L^2 - g(L)^{-1}) = \delta_1 k_1^{4/3} (g(L)^{-2} - L).
    \end{split}
\end{equation*}

Let's take a closer look at the last equation. By hypothesis $\delta_2 > \delta_1 k_1^{4/3}$, so just one of the following could hold:
\begin{itemize}
    \item $ L^2 - g(L)^{-1} > 0$ and $(L^2 - g(L)^{-1}) <  (g(L)^{-2} - L)$:\\
    \\
    from the second inequality we recover $(L^2 + L) <  (g(L)^{-2} + g(L)^{-1})$. Thanks to the properties we have already proved, it is not hard from the latter to deduce $ L < g(L)^{-1}$. It is now time to expand the right-hand side to obtain:
   \begin{equation*}
   \begin{split}
    &L < g(L)^{-1} = \frac{2\delta_2}{-\delta_1 k_1^{4/3} + \sqrt{\delta_1^2 k_1^{8/3} + 4\delta_1 \delta_2 k_1^{4/3} L^{-2} + 4\delta_2^2 L^{-1}}}\\
    & \iff 4\delta_1 \delta_2 k_1^{4/3} + 4\delta_2^2 L < 4\delta_2^2 + 4\delta_1 k_1^{4/2} L\\
    & \iff L (\delta_2 - \delta_1 k_1^{4/3}) < (\delta_2 - \delta_1 k_1^{4/3}) \iff L < 1
    \end{split}
    \end{equation*}
    that is absurd.
    \item $L^2 - g(L)^{-1} < 0$ and $(L^2 - g(L)^{-1}) >  (g(L)^{-2} - L)$:
    \\
    \\
    again, from our assumptions:
    \begin{equation*}
   \begin{split}
    & L^2 < g(L)^{-1} \iff L^2 < \frac{2\delta_2}{-\delta_1 k_1^{4/3} + \sqrt{\delta_1^2 k_1^{8/3} + 4\delta_1 \delta_2 k_1^{4/3} L^{-2} + 4\delta_2^2 L^{-1}}}\\
    & \iff \delta_1^2 k_1^{8/3} + 4\delta_1 \delta_2 k_1^{4/3} L^{-2} + 4\delta_2^2 L^{-1} < (2\delta_2 / L^2 + \delta_1 k_1^{4/3})^2\\
    & \iff 4\delta_1 \delta_2 k_1^{4/3} L^{-2} + 4\delta_2^2 L^{-1} < 4\delta_2^2 L^{-4} + 4\delta_1 \delta_2 k_1^{4/3}L^{-2}\\
    & \iff L^{3} < 1 \iff L < 1\\
    \end{split}
    \end{equation*}
    again against our assumption.
    \item $L^2 = g(L)^{-1}$ and $L = g(L)^{-2}$:
    \\
    \\
    if $x, y $ are two positive real numbers such that
    $$
    x^2 = y, \,\,\,\, x= y^2,
    $$
    the only solution is $x=y=1$.
\end{itemize}
 We conclude that $L = \lim_{n \rightarrow \infty} {b_{2n+1}} =1$. With same argument one can show $ \lim_{n \rightarrow \infty} {b_{2n}} = 1$, concluding the proof.
\end{proof}

We now state and prove an equivalent for Lemma \ref{lemmaavanti} when $\delta_1 / \delta_2 > k_1^{-4/3}$.\\
First, for every $N > 1$ consider the following recursive backward sequence $\{ (b^*_n)^{(N)} \}_n$:
\begin{equation}\label{ricorrenzaindietro}
\begin{split}
&(b^*_N)^{(N)} = C^* > 0, \\
&(b^*_{n})^{(N)} = \frac{-\delta_2 k_1^{-4/3} + \sqrt{\delta_2^2 k_1^{-8/3} + 4\delta_1 \delta_2 k_1^{-4/3} ((b^*_{n+1})^{(N)})^{-2} + 4\delta_1^2 ((b^*_{n+1})^{(N)})^{-1}}}{2\delta_1}
\end{split}
\end{equation}
for any $0 \leq n < N$, some positive starting value $C^*>0$ and some positive coefficient $\delta_1, \delta_2$ such that $\delta_1 / \delta_2 > k_1^{-4/3}$.
\\
Lemmas \ref{lemmaindietro} and \ref{lemmaindietrobis} tell useful information about sequence $\{ (b^*_n)^{(N)} \}_n$ and its asymptotic behaviour.

\begin{lem}\label{lemmaindietro}
For every starting value $C^*>0$, any $N>1$ and any positive coefficients $\delta_1$, $\delta_2$ such that $\delta_1 / \delta_2 > k_1^{-4/3}$, the recursive sequence $\{ (b_n^{*})^{(N)} \}_n$ defined above satisfies
\begin{enumerate}
\item $0 < (b_N^{*})^{(N)} < (b_n^{*})^{(N)} < (b_{N-1}^{*})^{(N)} < \frac{1}{C^{*}}$, if $C^* <1$, $\forall n\geq 1$.
\item $0 < (b_{N-1}^{*})^{(N)} < (b_n^{*})^{(N)} < (b_{N}^{*})^{(N)} = C^*$, if $C^* >1$, $\forall n\geq 1$.
\item $(b_n^{*})^{(N)} \equiv 1$, if $C^* =1$, $\forall n\geq 1$.
\end{enumerate}
\end{lem}
\begin{proof}
This proof is entirely equivalent to the one we proposed for Lemma \ref{lemmaavanti} by swapping $\delta_1$ and $\delta_2$ coefficients. The only statement left to prove is
$$
(b_{N-1}^{*})^{(N)} < \frac{1}{C^{*}}, \,\,\, \text{if} \,\,\, C^* <1.
$$
By a direct calculation we have
\begin{equation*}
\begin{split}
    (b_{N-1}^{*})^{(N)} < \frac{1}{C^{*}} \iff &\frac{-\delta_2 k_1^{-4/3} + \sqrt{\delta_2^2 k_1^{-8/3} + 4\delta_1 \delta_2 k_1^{-4/3} / C^{*2} + 4\delta_1^2 / C^{*}}}{2\delta_1} < \frac{1}{C^{*}}\\
    \iff & \delta_2^2 k_1^{-8/3} + \frac{4\delta_1 \delta_2 k_1^{-4/3}}{C^{*2}} + \frac{4\delta_1^2}{C^{*}} < \delta_2^2 k_1^{-8/3} + \frac{4\delta_1 \delta_2 k_1^{-4/3}}{C^{*}} + \frac{4\delta_1^2}{C^{*2}}\\
    \iff & \frac{\delta_2 k_1^{-4/3}}{C^{*2}} + \frac{\delta_1}{C^{*}} < \frac{\delta_2 k_1^{-4/3}}{C^{*}} + \frac{\delta_1}{C^{*2}}\\
    \iff & C^{*} (\delta_1 - \delta_2 k_1^{-4/3}) < (\delta_1 - \delta_2 k_1^{-4/3}) \iff C^{*}<1,
    \end{split}
\end{equation*}
due to the assumption $\delta_1 / \delta_2 > k_1^{-4/3}$.
\end{proof}

Lemma \ref{lemmaindietro} tells that for every $N>1$, $\{ (b^*_n) ^{(N)}\}$ lies in the compact set $[0, C^{*}]$ if $C^{*} \geq 1$ or $[0, 1/C^{*}]$ if $C^{*} < 1$, thus by compactness and a diagonal extraction argument we can choose a subsequence $(N_i)_i \in \mathbb{N}$ such that $(b_n^*)^{(N_i)}$ converges for all $n \in \mathbb{N}$ to some number $\tilde{b}_n^*$. The sequence $\tilde{b} = \{ \tilde{b}_n^* \}_n $ satisfies the following equation by construction
\begin{equation}\label{equazioneperlemmaindietrobis}
\tilde{b}^*_{n} = \frac{-\delta_2 k_1^{-4/3} + \sqrt{\delta_2^2 k_1^{-8/3} + 4\delta_1 \delta_2 k_1^{-4/3} (\tilde{b}^*_{n+1})^{-2} + 4\delta_1^2 (\tilde{b}^*_{n+1})^{-1}}}{2\delta_1}.
\end{equation}
\begin{lem}\label{lemmaindietrobis}
For every starting value $C^{*}>0$, and positive coefficient $\delta_1, \delta_2$ such that $\delta_1 / \delta_2 > k_1^{-4/3}$, the recursive sequence (\ref{equazioneperlemmaindietrobis}) satisfies
$$
\lim_{n \rightarrow \infty} \tilde{b}_n^{*} = 1.
$$
\end{lem}
\begin{proof}
It is equivalent to the proof of Lemma \ref{lemmaavanti} by swapping $\delta_1$ and $\delta_2$ coefficients.
\end{proof}
We are now ready to prove Theorem \ref{stationary}.
\\
\\
Let us start by considering equation (\ref{stationarysolution}) written in the following form
$$
0 = \delta_1 (a_{n-1}^2 - k_1 a_n a_{n+1}) - \delta_2 (a_{n+1}^2 - k_1^{-1} a_n a_{n-1}).
$$
We already focus our interest into positive solutions with no zero term, so dividing by $a_n$ both sides and changing variable with $b_n = \frac{a_n}{a_{n-1}}$ we obtain
$$
0 = \delta_1 (b_{n}^{-2} - k_1 b_{n+1}) - \delta_2 (b_{n+1}^2 - k_1^{-1} b_n^{-1}).
$$
We now apply a further change of variable $a_n = \tilde{a}_n / k_n^{1/3}$ and consequently $b_n = \tilde{b}_n / k_1^{1/3}$ to finally get
$$
0 = \delta_1 (k_1^{2/3}\tilde{b}_{n}^{-2} - k_1^{2/3} \tilde{b}_{n+1}) - \delta_2 (k_1^{-2/3}\tilde{b}_{n+1}^2 - k_1^{-2/3} \tilde{b}_n^{-1}).
$$
We can solve the above equation of degree two restricting ourselves only to positive solutions
\begin{equation}\label{ricorrenzaavantidim}
\begin{split}
&\tilde{b}_0 = a_1 / F > 0, \\
&\tilde{b}_{n+1} = \frac{-\delta_1 k_1^{4/3} + \sqrt{\delta_1^2 k_1^{8/3} + 4\delta_1 \delta_2 k_1^{4/3} \tilde{b}_n^{-2} + 4\delta_2^2 \tilde{b}_n^{-1}}}{2\delta_2}.
\end{split}
\end{equation}
Lemma \ref{lemmaavanti} shows that $\lim_{n \rightarrow \infty} \tilde{b}_n = 1$ every time $\delta_1 / \delta_2 < k_1^{-4/3}$. Thus, $\lim_{n \rightarrow \infty} b_n = k_1^{-1/3}$ and $\lim_{n \rightarrow \infty} a_n / a_{n-1} = k_1^{-1/3} < 1$, proving Theorem \ref{stationary} in the case $\delta_1 / \delta_2 < k_1^{-4/3}$.
\\
In the same fashion, one can consider a backward change of variable $b_n = \frac{a_{n-1}}{a_{n}}$ and obtain
$$
0 = \delta_1 (b_{n}^{2} - k_1 b_{n+1}^{-1}) - \delta_2 (b_{n+1}^{-2} - k_1^{-1} b_n).
$$
We now apply a further change of variable $a_n = \tilde{a}_n / k_n^{1/3}$ and consequently $b_n = \tilde{b}_n / k_1^{1/3}$ and finally get
$$
0 = \delta_1 (k_1^{2/3}\tilde{b}_{n}^{2} - k_1^{2/3} \tilde{b}_{n+1}^{-1}) - \delta_2 (k_1^{-2/3}\tilde{b}_{n+1}^{-2} - k_1^{-2/3} \tilde{b}_n).
$$
As before we can solve the above equation of degree two restricting ourselves only to positive solution
\begin{equation}\label{ricorrenzaindietrodim}
\begin{split}
\tilde{b}_0 &= a_0 / a_1 > 0\\
\tilde{b}_{n} &= \frac{-\delta_2 k_1^{-4/3} + \sqrt{\delta_2^2 k_1^{-8/3} + 4\delta_1 \delta_2 k_1^{-4/3} \tilde{b}_{n+1}^{-2} + 4\delta_1^2 \tilde{b}_{n+1}^{-1}}}{2\delta_1}.
\end{split}
\end{equation}
Lemma \ref{lemmaindietro} states $\lim_{n \rightarrow \infty} \tilde{b}_n = 1$ every time $\delta_1 / \delta_2 > k_1^{-4/3}$. Thus $\lim_{n \rightarrow \infty} b_n = k_1^{-1/3}$ and $\lim_{n \rightarrow \infty} a_{n-1} / a_{n} = k_1^{1/3} > 1$, proving Theorem \ref{stationary} also in the case $\delta_1 / \delta_2 > k_1^{-4/3}$.

\section{Proof of Theorem \ref{teoselfsimilarmix}}\label{dimostrazione3}

Now that we have successfully proved Theorem \ref{stationary}, we observe that equation (\ref{selfsimilarsolution}) for self-similar sequences differs from equation (\ref{stationarysolution}) for constant solution only by a perturbation term $\frac{a_n}{k_n}$. Thus, we adapt our proof to take care of this extra term.
\\
\\
Indeed, without loss of generality, we can set $a_0 = 0$ and $a_n > 0$ for every $n>0$ in equation (\ref{selfsimilarsolution}), then dividing by $a_n$ both sides and changing variable with $b_n = \frac{a_n}{a_{n-1}}$ we obtain
$$
\delta_1 (b_{n}^{-2} - k_1 b_{n+1}) - \delta_2 (b_{n+1}^2 - k_1^{-1} b_n^{-1}) - \frac{1}{a_n k_n} = 0.
$$
We now apply a further change of variable $a_n = \tilde{a}_n / k_n^{1/3}$ and consequently $b_n = \tilde{b}_n / k_1^{1/3}$ and finally get
$$
\delta_1 (k_1^{2/3}\tilde{b}_{n}^{-2} - k_1^{2/3} \tilde{b}_{n+1}) - \delta_2 (k_1^{-2/3}\tilde{b}_{n+1}^2 - k_1^{-2/3} \tilde{b}_n^{-1}) - \frac{\tilde{a}_n}{k_n^{2/3}} = 0.
$$
We can solve the above equation of degree two restricting ourselves only to positive solution
\begin{equation}\label{ricorrenzaavantidimselfsimilar}
\begin{split}
&\tilde{b}_1 = a_2 / a_1 > 0, \\
&\tilde{b}_{n+1} = \frac{-\delta_1 k_1^{4/3} + \sqrt{\delta_1^2 k_1^{8/3} + 4\delta_1 \delta_2 k_1^{4/3} \tilde{b}_n^{-2} + 4\delta_2^2 \tilde{b}_n^{-1} + 4\delta_2 k_1^{2/3} \epsilon_n}}{2\delta_2},
\end{split}
\end{equation}
where $\epsilon_n = (\tilde{a}_n k_n^{2/3})^{-1} = (a_n k_n)^{-1}$.\\

In the next Lemma we prove that $\lim _{n \rightarrow \infty} \epsilon_n = 0$.
\begin{lem}\label{perturbazione}
If $\{ a_n \}_n$ is a positive self-similar sequence satisfying equation (\ref{selfsimilar}) with $k_1^{-4} \leq \delta_1 / \delta_2 \leq 1$, then
$$
\lim_{n \rightarrow \infty } a_n k_n = \infty.
$$
\end{lem}
\begin{proof}
Let's consider a change of variable $a_n = c_n / k_n$ in equation (\ref{selfsimilar}) to obtain
$$
-c_n = \delta_1(k_1^2 c_{n-1}^2 - c_n c_{n+1}) -\delta_2 (k_1^{-2}c_{n+1}^2 - c_n c_{n-1}).
$$
It is helpful to express $c_{n+1}$ as function of previous terms
$$
c_{n+1} = \frac{-\delta_1 c_n + \sqrt{\delta_1^2 c_n^2 + 4\delta_2 k_1^{-2}c_n + 4\delta_1 \delta_2 c_{n-1}^2 + 4\delta_2^2 k_1^{-2} c_n c_{n-1}}}{2\delta_2 k_1^{-2}}.
$$
We first prove that $c_{n+1} > c_{n-1}$ for every $n \geq 1$.
\\
\\
Indeed, we have
\begin{equation}\label{remark}
    \begin{split}
        &c_{n+1} = \frac{-\delta_1 c_n + \sqrt{\delta_1^2 c_n^2 + 4\delta_2 k_1^{-2}c_n + 4\delta_1 \delta_2 c_{n-1}^2 + 4\delta_2^2 k_1^{-2} c_n c_{n-1}}}{2\delta_2 k_1^{-2}} > c_{n-1}\\
        & \iff 4\delta_2 k_1^{-2}c_n + 4\delta_1 \delta_2 c_{n-1}^2 + 4\delta_2^2 k_1^{-2} c_n c_{n-1} > 4\delta_2^2 k_1^{-4} c_{n-1}^2 + 4\delta_1 \delta_2 k_1^{-2} c_n c_{n-1}\\
        & \iff k_1^{-2}c_n + \delta_1 c_{n-1}^2 + \delta_2 k_1^{-2} c_n c_{n-1} > \delta_2 k_1^{-4} c_{n-1}^2 + \delta_1 k_1^{-2} c_n c_{n-1}\\
        & \iff k_1^{-2}c_n + c_{n-1}^2 (\delta_1 - \delta_2 k_1^{-4}) + k_1^{-2} c_n c_{n-1} (\delta_2 - \delta_1) > 0,
    \end{split}
\end{equation}
and the last inequality holds thanks to the assumption $\delta_2 k_1^{-4} \leq \delta_1 \leq \delta_2$.
\\
\\
Thus, there is a positive value $D>0$ so that $c_n \geq D$ for every $n\geq 1$.
\\
\\
With a similar argument it is possible to prove the existence of $M>1$ so that $c_{n+1} > c_{n-1} \cdot M$, which will conclude the proof.
\\
Indeed,
\begin{equation*}
    \begin{split}
        &c_{n+1} = \frac{-\delta_1 c_n + \sqrt{\delta_1^2 c_n^2 + 4\delta_2 k_1^{-2}c_n + 4\delta_1 \delta_2 c_{n-1}^2 + 4\delta_2^2 k_1^{-2} c_n c_{n-1}}}{2\delta_2 k_1^{-2}} > c_{n-1} \cdot M\\
        & \iff k_1^{-2}c_n + \delta_1 c_{n-1}^2 + \delta_2 k_1^{-2} c_n c_{n-1} > \delta_2 k_1^{-4} M^2 c_{n-1}^2 + \delta_1 k_1^{-2} M c_n c_{n-1}.
    \end{split}
\end{equation*}
Last inequality further simplifies as follows
\begin{equation*}
\begin{split}
        &k_1^{-2}c_n + c_{n-1}^2 (\delta_1 - \delta_2 k_1^{-4}M^2) + k_1^{-2} c_n c_{n-1} (\delta_2 - \delta_1 M) > 0\\
        & \iff k_1^{-2}D + D^2 (\delta_1 + k_1^{-2} \delta_2)(1 - k_1^{-2}M) > 0.
        \end{split}
\end{equation*}

Finally, by hypothesis $\lambda >1$ and $\beta>0$, hence it is possible to choose $1 < M \leq k_1^{2}$ in the latter relation, so that left-hand sides becomes a sum of positive term.
\end{proof}

\begin{rem}
We notice that upper and lower bounds on Theorem \ref{teoselfsimilarmix} arise from inequality (\ref{remark}). Condition $\delta_2 k_1^{-4} \leq \delta_1 \leq \delta_2$ is sufficient in order to satisfy inequality (\ref{remark}), although the first term $k_1^{-2}c_n$ gives a small but significant positive contribute. Consequently, upper and lower bounds for the ratio $\delta_1 / \delta_2$ can be further refined. Numerical simulation suggests the existence of a \emph{true} bound $L_{true}< k_1^{-4}$ such that theorem \ref{teoselfsimilarmix} holds in the wider domain  $\delta_2 \cdot L_{true} \leq \delta_1 $. This is consistent with Theorem \ref{teoselfsimilar} where $\delta_1 = 1$ and $\delta_2 = 0$ but also with Theorem 1 in \cite{jeong} where the author proved existence of self-similar solutions for $\delta_1 = 1$ and $0 \leq \delta_2<\beta_0$ for some small $\beta_0 > 0$.

\end{rem}

The following Lemma is an equivalent of Lemma \ref{lemmaavanti} for the sequence (\ref{ricorrenzaavantidimselfsimilar}).

\begin{lem}\label{lemmaavantiselfsimilar}
Let's suppose $\tilde{b}_1 = a_2 / a_1 = C >0$, and $k_1^{-4} \leq \delta_1/\delta_2 \leq k_1^{-4/3}$. If $\tilde{b}_1 \geq \tilde{b}_3$ then
\begin{itemize}
    \item $C \geq \tilde{b}_3 \geq \tilde{b}_5 \geq \ldots \geq \tilde{b}_{2n+1} \geq \ldots > 1$ for all $n\geq 0$;
    \item $0 < \tilde{b}_2 \leq \tilde{b}_4 \leq \tilde{b}_6 \leq \ldots \tilde{b}_{2n} \leq \ldots \leq 1 + \sqrt{\epsilon_1 \frac{k_1^{2/3}}{\delta_2}}$ for all $n \geq 1$.
\end{itemize}
Otherwise, if $\tilde{b}_1 < \tilde{b}_3$ then
\begin{itemize}
    \item $C \leq \tilde{b}_3 \leq \tilde{b}_5 \leq \ldots \leq \tilde{b}_{2n+1} \leq \ldots \leq 1 + \sqrt{\epsilon_2 \frac{k_1^{2/3}}{\delta_2}}$ for all $n\geq 0$;
    \item $\tilde{b}_2 \geq \tilde{b}_4 \geq \tilde{b}_6 \geq \ldots \tilde{b}_{2n} \geq \ldots > 1$ for all $n \geq 1$.
\end{itemize}
Moreover, $\lim_{n \rightarrow \infty} \tilde{b}_{2n+1} = \lim_{n \rightarrow \infty} \tilde{b}_{2n} = 1$.
\end{lem}
\begin{proof}
We consider only the case $\tilde{b}_1 \geq \tilde{b}_3$ (the other being specular).
\\
Let's first observe that
\begin{equation*}
\begin{split}
\tilde{b}_4 = &\frac{-\delta_1 k_1^{4/3} + \sqrt{\delta_1^2 k_1^{8/3} + 4\delta_1 \delta_2 k_1^{4/3} \tilde{b}_3^{-2} + 4\delta_2^2 \tilde{b}_3^{-1} + 4\delta_2 k_1^{2/3} \epsilon_3}}{2\delta_2}\\
& \leq \frac{-\delta_1 k_1^{4/3} + \sqrt{\delta_1^2 k_1^{8/3} + 4\delta_1 \delta_2 k_1^{4/3} \tilde{b}_3^{-2} + 4\delta_2^2 \tilde{b}_3^{-1} + 4\delta_2 k_1^{2/3} \epsilon_1}}{2\delta_2}
\end{split}
\end{equation*}
thus $\tilde{b}_2 \leq \tilde{b}_4$ if and only if
\begin{equation*}
\begin{split}
&\sqrt{\delta_1^2 k_1^{8/3} + 4\delta_1 \delta_2 k_1^{4/3} \tilde{b}_1^{-2} + 4\delta_2^2 \tilde{b}_1^{-1} + 4\delta_2 k_1^{2/3} \epsilon_1} \\&\leq \sqrt{\delta_1^2 k_1^{8/3} + 4\delta_1 \delta_2 k_1^{4/3} \tilde{b}_3^{-2} + 4\delta_2^2 \tilde{b}_3^{-1} + 4\delta_2 k_1^{2/3} \epsilon_1}\\
& \iff 4\delta_1 \delta_2 k_1^{4/3} \tilde{b}_1^{-2} + 4\delta_2^2 \tilde{b}_1^{-1} \leq 4\delta_1 \delta_2 k_1^{4/3} \tilde{b}_3^{-2} + 4\delta_2^2 \tilde{b}_3^{-1}\\
& \iff \delta_1 k_1^{4/3} \tilde{b}_1^{-2} + \delta_2 \tilde{b}_1^{-1} \leq \delta_1 k_1^{4/3} \tilde{b}_3^{-2} + \delta_2 \tilde{b}_3^{-1}\\
\end{split}
\end{equation*}
and, remembering $\delta_1 / \delta_2 \leq k_1^{-4/3}$, the latter is implied by $\tilde{b}_1 \geq \tilde{b}_3$.\\
The following cascade of implications is then an immediate consequence
$$
\tilde{b}_1 \geq \tilde{b}_3 \implies \tilde{b}_2 \leq \tilde{b}_4 \implies \tilde{b}_3 \geq \tilde{b}_5 \implies \tilde{b}_4 \leq \tilde{b}_6 \implies \ldots
$$
$$
\implies \tilde{b}_{2n-1} \geq \tilde{b}_{2n+1} \implies \tilde{b}_{2n} \leq \tilde{b}_{2n+2}, \,\,\,\ \forall n\geq 1.
$$
We now say that $\{ \tilde{b}_{2n+1}\}_n$ admits a finite limit, say $L_1$: with the same argument used in Lemma \ref{lemmaavanti}, thanks to Lemma \ref{perturbazione} one can easily prove $L_1 = 1$.
\\
We now prove the upper bound
$$
\tilde{b}_{2n} \leq 1 + \sqrt{\epsilon_1 \frac{k_1^{2/3}}{\delta_2}}, \,\,\,\ \forall n\geq 1.
$$
We first stress that $\tilde{b}_1 \geq \tilde{b}_3$ only if $C = \tilde{b}_1 > 1$, thus we can write
\begin{equation*}
\begin{split}
\tilde{b}_{2n} =& \frac{-\delta_1 k_1^{4/3} + \sqrt{\delta_1^2 k_1^{8/3} + 4\delta_1 \delta_2 k_1^{4/3} \tilde{b}_{2n-1}^{-2} + 4\delta_2^2 \tilde{b}_{2n-1}^{-1} + 4\delta_2 k_1^{2/3} \epsilon_{2n-1}}}{2\delta_2}\\
\leq &\frac{-\delta_1 k_1^{4/3} + \sqrt{\delta_1^2 k_1^{8/3} + 4\delta_1 \delta_2 k_1^{4/3} \tilde{b}_{2n-1}^{-2} + 4\delta_2^2 \tilde{b}_{2n-1}^{-1}} + \sqrt{4\delta_2 k_1^{2/3} \epsilon_{2n-1}}}{2\delta_2}\\
\leq& \frac{-\delta_1 k_1^{4/3} + \sqrt{\delta_1^2 k_1^{8/3} + 4\delta_1 \delta_2 k_1^{4/3} \tilde{b}_{2n-1}^{-2} + 4\delta_2^2 \tilde{b}_{2n-1}^{-1}} + \sqrt{4\delta_2 k_1^{2/3} \epsilon_{1}}}{2\delta_2}\\
\leq& \,1 + \sqrt{\epsilon_1 \frac{k_1^{2/3}}{\delta_2}}
\end{split}
\end{equation*}
where the last inequality is a direct consequence of $C>1$ and $b_{2n-1} > L_1=1$.
\\
We know also that $\{ \tilde{b}_{2n}\}_n$ admit finite limit, say $L_2$: with similar argument used in Lemma \ref{lemmaavanti} and thanks to Lemma \ref{perturbazione} we conclude $L_2=1$.
\end{proof}

Lemmas \ref{lemmaavanti} and \ref{lemmaavantiselfsimilar} show that $\lim_{n \rightarrow \infty} \tilde{b}_n = 1$ every time $k_1^{-4} \leq \delta_1 / \delta_2 \leq k_1^{-4/3}$. Thus $\lim_{n \rightarrow \infty} b_n = k_1^{-1/3}$ and $\lim_{n \rightarrow \infty} a_n / a_{n-1} = k_1^{-1/3} < 1$, proving Theorem \ref{teoselfsimilarmix} in the case $k_1^{-4} \leq \delta_1 / \delta_2 \leq k_1^{-4/3}$.
\\
\\
We now mimic again the proof of Theorem \ref{stationary} also in the case $k_1^{-4/3} < \delta_1 / \delta_2 \leq 1$, by considering a backward change of variable $b_n = \frac{a_{n-1}}{a_{n}}$ to obtain
$$
-\frac{1}{a_n k_n} = \delta_1 (b_{n}^{2} - k_1 b_{n+1}^{-1}) - \delta_2 (b_{n+1}^{-2} - k_1^{-1} b_n).
$$
By applying a further change of variable $a_n = \tilde{a}_n / k_n^{1/3}$ and consequently $b_n = \tilde{b}_n / k_1^{1/3}$ we finally get
$$
-\frac{1}{\tilde{a}_n k_n^{2/3}}  = \delta_1 (k_1^{2/3}\tilde{b}_{n}^{2} - k_1^{2/3} \tilde{b}_{n+1}^{-1}) - \delta_2 (k_1^{-2/3}\tilde{b}_{n+1}^{-2} - k_1^{-2/3} \tilde{b}_n).
$$
As before we can solve the above backward equation of degree two restricting ourselves only to positive solutions. For every $N>1$ let be
\begin{equation*}
\begin{split}
&\tilde{b}_{N}^{(N)} = C^* > 0, \\
&\tilde{b}_{n}^{(N)} = \frac{-\delta_2 k_1^{-4/3} + \sqrt{\delta_2^2 k_1^{-8/3} + 4\delta_1 \delta_2 k_1^{-4/3} (\tilde{b}_{n+1}^{(N)})^{-2} + 4\delta_1^2 (\tilde{b}_{n+1}^{(N)})^{-1} - 4 \delta_1 \epsilon_n^*}}{2\delta_1},
\end{split}
\end{equation*}
for any $0\leq n < N$, with $\epsilon_n^* = \frac{1}{\tilde{a}_n^{(N)} k_n^{2/3}}= \frac{1}{a_n^{(N)} k_n}$.
\\
\\
Lemma \ref{lemmaavantibackward} shows that there is $C^* >0$ so that the sequence $\{ \tilde{b}_n^{(N)} \}_n$ is well-defined and lies uniformly in a compact set for every $N>0$.
\begin{lem}\label{lemmaavantibackward}
For every $k_1^{-4/3} < \delta_1 / \delta_2 \leq 1$, there is $C^*>0$ so that $\tilde{b}_{n}^{(N)}$ is well-defined for every $0<n\leq N$.
\\
Moreover, there is $M^*>1$ and $N^*>0$ so that the sequence $\{ \tilde{b}_n^{(N)} \}_n$ satisfies
$$
0 < \frac{1}{M^*} \leq \tilde{b}_n^{(N)} \leq M^*
$$
for every $N > N^*$ and every $N^*< n\leq N$.
\end{lem}
\begin{proof}
By definition $\frac{\tilde{a_{N-1}^{(N)}}}{\tilde{a_{N}^{(N)}}} = C^*>0$, and by Lemma \ref{perturbazione} we can choose the free parameter $C^*$ small enough so that
$$
4\delta_1 \epsilon_n^* \leq 4\delta_1 \max \{ \epsilon_1^*, \epsilon_2^* \} \leq \delta_2 k_1^{-8/3}
$$
for every $n>0$, this implies that the square root in the expression of sequence $\{ \tilde{b}_n^{(N)} \}_n$ is well-defined.
\\
\\
We now prove the statement by induction over $n\leq N$. Indeed, for every $C^*$ there is $M^*>1$ so that 
$$
\frac{1}{M^*} \leq C^*=\tilde{b}_N^{(N)} \leq M^*
$$
Let's suppose now $\frac{1}{M^*} \leq \tilde{b}_n^{(N)} \leq M^*$ for some $n\leq N$. By definition
\begin{equation*}
    \begin{split}
        \tilde{b}_{n-1}^{(N)} \leq M^* &\iff \frac{4 \delta_1 \delta_2 k_1^{-4/3}}{(\tilde{b}_n^{(N)})^2} + \frac{4 \delta_1^2}{\tilde{b}_n^{(N)}} \leq 4\delta_1^2M^{*2} + 4\delta_1 \delta_2 k_1^{-4/3}M^* + 4\delta_1 \epsilon_n^*\\
        & \iff \frac{\delta_2 k_1^{-4/3}}{(\tilde{b}_n^{(N)})^2} + \frac{ \delta_1}{\tilde{b}_n^{(N)}} \leq \delta_1 M^{*2} + \delta_2 k_1^{-4/3}M^* + \epsilon_n^*.
    \end{split}
\end{equation*}
By hypothesis $\frac{1}{M^*} \leq \frac{1}{\tilde{b}_n^{(N)}} \leq M^*$, so it is enough to prove
\begin{equation*}
    \begin{split}
        &\delta_2 k_1^{-4/3}M^{*2}+ \delta_1 M^{*2} \leq \delta_1 M^{*2} + \delta_2 k_1^{-4/3}M^*\\
        & \iff \delta_2 k_1^{-4/3} (M^*-1) \leq \delta_1 (M^*-1),
    \end{split}
\end{equation*}
the latter being true thanks to $M^*>1$ and $\delta_1 > \delta_2 k_1^{-4/3}$. This proves the right side of our claim.
\\
Again, by definition
\begin{equation*}
    \begin{split}
        \tilde{b}_{n-1}^{(N)} \geq \frac{1}{M^*} &\iff \frac{4 \delta_1 \delta_2 k_1^{-4/3}}{(\tilde{b}_n^{(N)})^2} + \frac{4 \delta_1^2}{\tilde{b}_n^{(N)}} \geq \frac{4\delta_1^2}{M^{*2}} + \frac{4\delta_1 \delta_2 k_1^{-4/3}}{M^*} + 4\delta_1 \epsilon_n^*\\
        & \iff \frac{\delta_2 k_1^{-4/3}}{(\tilde{b}_n^{(N)})^2} + \frac{ \delta_1}{\tilde{b}_n^{(N)}} \geq \frac{\delta_1}{M^{*2}} + \frac{\delta_2 k_1^{-4/3}}{M^*} + \epsilon_n^*.
    \end{split}
\end{equation*}
We observe that it is enough to prove
\begin{equation*}
    \begin{split}
        & \frac{\delta_2 k_1^{-4/3}}{M^{*2}} + \frac{ \delta_1}{M^*} \geq \frac{\delta_1}{M^{*2}} + \frac{\delta_2 k_1^{-4/3}}{M^*} + \epsilon_n^*\\
        & \iff \delta_2 k_1^{-4/3} + \delta_1 M^* \geq \delta_1 + \delta_2 k_1^{-4/3}M^* + \epsilon_n^* M^{*2}\\
        & \iff (\delta_1 - \delta_2 k_1^{-4/3}) \frac{M^*-1}{M^{*2}} \geq \epsilon_n^*.
    \end{split}
\end{equation*}
By Lemma \ref{perturbazione}, $\lim_{n \rightarrow \infty} \epsilon_n^* = 0$, so there is $N^*>0$ so that
$$
(\delta_1 - \delta_2 k_1^{-4/3}) \frac{M^*-1}{M^{*2}} \geq \epsilon_n^*
$$
for every $N^*< n\leq N$.

\end{proof}

Lemma \ref{lemmaavantibackward} shows that exists $N^{*}$ such that for every $N>0$, $\{ \tilde{b}_n ^{(N)}\}$ lies in a compact set, thus by compactness and a diagonal extraction argument we can choose a subsequence $(N_i)_i \in \mathbb{N}$ such that $\tilde{b}_n^{(N_i)}$ converges for all $n \in \mathbb{N}$ to some number $\tilde{b}^*_n$. The sequence $\tilde{b}^* = \{ \tilde{b}^*_n \}_n $ satisfies the following equation by construction
\begin{equation*}
\tilde{b}^*_{n} = \frac{-\delta_2 k_1^{-4/3} + \sqrt{\delta_2^2 k_1^{-8/3} + 4\delta_1 \delta_2 k_1^{-4/3} (\tilde{b}^*_{n+1})^{-2} + 4\delta_1^2 (\tilde{b}^*_{n+1})^{-1} - 4\delta_1 \epsilon_n^*}}{2\delta_1}.
\end{equation*}
\\
Finally, by the same argument used in Lemmas \ref{lemmaindietro} and \ref{perturbazione} we deduce $\lim_{n \rightarrow \infty} \tilde{b}^*_n = 1$ every time $\delta_1 / \delta_2 \geq k_1^{-4/3}$. Thus, the related solution $a_n^{*}$ satisfies $\lim_{n \rightarrow \infty} \tilde{b}^*_n = a_{n-1}^{*} / a_{n}^{*} = k_1^{1/3} > 1$. Moreover, by a continuity argument similar to Proposition \ref{proptecnica} we can pick a suitable $C^{*}$ so that $b_0 = 0$ and consequently $a_0 = 0$, proving Theorem \ref{teoselfsimilarmix} statements also in the case $k_1^{-4/3} < \delta_1 / \delta_2 \leq 1$.

\section*{Acknowledgements}

The author thanks Prof. F. Morandin, for suggesting the problem, reviewing drafts of the work and for helpful discussions, Prof. F. Flandoli and Prof. L.A. Bianchi for their insightful comments and review.

\addcontentsline{toc}{chapter}{Bibliography}

\bibliographystyle{elsarticle-num}

\end{document}